\documentclass[11pt,reqno]{amsart}
\usepackage{amsmath,amssymb}

\usepackage{graphicx}

\long\def\beginpgfgraphicnamed#1#2\endpgfgraphicnamed{\includegraphics{#1}}
\usepackage[unicode,bookmarks,colorlinks]{hyperref}

\numberwithin{equation}{section}
\newtheorem{theorem}{Theorem}[section]

\newtheorem{proposition}[theorem]{Proposition}
\newtheorem{corollary}[theorem]{Corollary}

\newtheorem{remark}[theorem]{Remark}

\theoremstyle{definition}

\newcommand{\R}{\mathbb{R}}

\newcommand{\A}{\mathbb{A}}

\newcommand{\bb}[1]{\mathbb{#1}}
\newcommand{\bR}{\bb{R}}
\newcommand{\bI}{\bb{I}}

\theoremstyle{definition}
 
\date{}

\title[Symmetrization and L\'evy Processes]{Symmetrization  of L\'evy processes and applications}
\author{Rodrigo Ba\~nuelos}\thanks{R. Ba\~nuelos was supported in part  by NSF Grant
\# 0603701-DMS}
\address{Department of Mathematics, Purdue University, West Lafayette, IN 47907}
\email{banuelos@math.purdue.edu}
\author{Pedro J. M\'endez-Hern\'andez}\thanks{P.J. M\'endez-Hern\'andez was supported in part by project $N^{0}$ 821-A7-177 of Centro de Investigaci\'on en Matem\'atica Pura y Aplicada (CIMPA)}
\address{Escuela de Matem\'atica,
 Universidad de Costa Rica,
San Jos\'e, Costa Rica}
\email{pedro.mendez@ucr.ac.cr}

\begin{document}
\maketitle

\begin{abstract}
  \noindent  It is shown that many of the classical generalized isoperimetric inequalities for the Laplacian when viewed in terms of Brownian motion extend to a wide class of L\'evy processes. The results are derived from the multiple integral inequalities of Brascamp, Lieb and Luttinger but  the probabilistic structure of the processes plays a crucial role in the proofs. 

\end{abstract}

\section{Introduction}

Let $D$ be an open connected set in $\bR^{d}$ of finite Lebesgue measure. Henceforth we shall refer to such sets simply as domains. We will  denote by $D^*$   the open ball in $\bR^{d}$ centered
at the origin $0$ with the same Lebesgue measure as $D$, and $|D|$ will denote the Lebesgue measure of $D$. There is a large class of quantities which are related to Brownian motion killed  upon leaving $D$ that are maximized, or minimized, by the corresponding quantities  for $D^*$. Such results often go by the name of  {\it generalized isoperimetric inequalities}. They  include the celebrated Rayleigh-Faber-Krahn inequality on the first eigenvalue of 
the Dirichlet Laplacian, inequalities for transition densities (heat kernels), Green functions, and electrostatic capacities (see \cite{bandle},  \cite{luttinger1}, \cite{luttinger2},  \cite{luttinger3} and   \cite{polya-szego}). 

Many of these {\it isoperimetric inequalities} can be beautifully  formulated  in terms of exit times of  the Brownian motion $B_t$ from the domain $D$. For example, if $\tau_D$ is the first exit time of $B_t$ from $D$, then for all $x \in D$   
\begin{equation} P^x \left\{ \,\tau_D > 0 \,\right\} \leq P^0 \left\{ \,\tau_{D^*} > 0 \,\right\}
,\label{exitbrownian}
\end{equation} 
where $0$ is the origin of $\bR^d$. Inequality (\ref{exitbrownian}) contains not only the classical  Rayleigh-Faber-Krahn inequality but inequalities for heat kernels and Green functions as well.  This inequality is now classical and can be found in many places in the literature.  For one of its first occurrences, using the Brascamp-Lieb-Luttinger multiple integrals techniques,  please see Aizenman and Simon \cite{IS}.  Similar inequalities can be obtained by these methods for domains of fixed inradius rather than fixed volume.  For more on this, we refer the reader to \cite{banuelos} and \cite{mendez}.  Also, versions of some of these results hold for Brownian motion on spheres and hyperbolic spaces, see \cite{BS} and references therein. 

Once these  isoperimetric-type inequalities are formulated in terms of exit times of Brownian motion, it is completely natural to enquire as to their validity  for other stochastic processes, and particularly for more general  L\'evy processes whose generators, as pseudo differential operators,  are natural extensions of the Laplacian. Such extensions have been obtained in recent years for the so called  ``symmetric stable processes" in $\bR^d$ and for more general processes obtained from subordination of Brownian motion. We refer the reader to  \cite{banuelos}, \cite{betsakos}, \cite{mendez}, \cite{simon}. 

The purpose of this paper is to show that many of these results continue to hold for very general L\'evy processes. At the heart of
these extensions are the  rearrangement inequalities of Brascamp, Lieb and Luttinger \cite{brascamp}. However, the probabilistic structure of
L\'evy processes enters in a very crucial way. Of particular importance for our method is the fact, derived from the L\'evy-Khintchine formula,
that our processes are weak limits of sums of a compound Poisson process and a Gaussian process.

We begin with a general description of L\'evy processes. A L\'evy process  $X_t$ in $\bR^d$ is  a stochastic process with independent and stationary 
increments which is ``stochastically" continuous. That is, for all $0<s< t<\infty$, $A\subset \bR^d$,  
\[P^{x}\{\,X_t-X_s\in A\,\}=P^{0}\{\,X_{t-s}\in A\,\},\] 
for any given sequence  of  ordered times $0<t_1<t_2<\cdots <t_m<\infty,$ the random variables  $X_{t_1}-X_0,\,\,  X_{t_2}-X_{t_1}, \dots, X_{t_m}-X_{t_{m-1}}$
are independent, and  for all $\varepsilon>0$, 
\[\lim_{t\to s}P^{x}\left\{\,|X_t-X_s |>\varepsilon\,\right\}=0.\]

The celebrated L\'evy-Khintchine formula
\cite{sato} guarantees the existence of a triple $\left( b , \A, \nu \right)$ such that the characteristic function of the process is given by 
\begin{equation}    E^x\left[\, e^{i \xi \cdot X_t}\,\right]  = e^{-t
\Psi(\xi)+i \xi \cdot x},\label{levyX} \end{equation}
where 
\begin{equation*} 
\Psi(\xi)= -i \langle b, \xi \rangle + \frac{1}{2}\langle \A \cdot \xi, \xi \rangle +
\int_{\bR^d}\left[\, 1 + i \langle \xi,y \rangle\,\bI_{B} - e^{i\,\xi \cdot y} \,\right]\,
d\nu(y).
\end{equation*}
Here,  $b \in \bR^d$,   $\A$ is a nonnegative $d \times d$ symmetric matrix, 
$\bI_B$ is the indicator function of the ball $B$ centered at the origin of radius 1, and $\nu$ is a measure  on $\bR^d$ such that 
\begin{equation}
\int_{\bR^d} \frac{|y|^2}{1+|y|^2} \,d\nu(y) < \infty \, \text{ and }\nu\left(\, \{ 0\}\, \right)=0.\label{levy-measure}
\end{equation}
The triple $\left( b , \A, \nu \right)$ is called the characteristics of  the process and  the measure $\nu$ is called the L\'evy measure of the process. Conversely, given a triple $\left( b , \A, \nu \right)$ with such properties  there is  L\'evy processes corresponding to it.  We will use the fact that any L\'evy process has a version with  paths that  are right continuous with left limits, so called ``c\`adl\`ag" paths. 

Next we recall the basic facts on symmetrization needed to state our results, more details on the properties of symmetrization  used in this paper can be found in the appendix in Section 6. Given a positive measurable function $f$, its symmetric decreasing rearrangement $f^*$ is the  unique  function satisfying 
\[
f^*(x)=f^*(y), \text{ if } |x|=|y|,\,\]
\[ f^*(x) \leq f^*(y), \text{ if } |x| \geq  |y|,\] \[ \lim_{|x| \to |y|^+}f^*(x)=f^*(y),
\]
and  
\begin{equation} 
m\left\{ f >  t \right \}= m\left\{ f^* >  t \right\}\label{phi_n},
\end{equation}
for all $ t \geq 0$.  Following \cite{lieb-loss}, under the assumption that $f$ vanishes at infinity,  an explicit expression for this function  is:

\begin{equation*}
f^*(x)=\int_0^{\infty} \chi^{*}_{\{|f|>t\}}(x) \, dt.
\end{equation*}
This explicit  representation is used only at the end of section \S4 in the case that $f$ is the indicator function of an open set  of finite area .

For symmetrization purposes, in this paper we will only consider L\'evy  measures  $\nu$ that are absolutely continuous with
respect to  the Lebesgue measure $m$. It may be that some of the results in this paper hold for more general L\'evy processes but at this stage we are not able to go beyond the absolute continuity case. 
Let $\phi$ be the density of $\nu$ and $\phi^*$ be its symmetric decreasing rearrangement. Since the function 
\[\psi(y)= 1-\frac{ |y|^2}{1+|y|^2} \]
is a positive, decreasing and radially symmetric, that is, $\psi^*=\psi$,   it follows that (see Theorem 3.4 in \cite{lieb-loss}) 
\begin{equation}
\int_{\bR^d} \,\frac{|y|^2}{1+|y|^2}\,\phi^*(y)\, dy\,\leq \,
\int_{\bR^d} \,\frac{|y|^2}{1+|y|^2}\,\phi(y)\, dy\, < \infty \,. 
\end{equation}
Hence the measure $\phi^*(y)\,dy$ satisfies (\ref{levy-measure}) and it is also a L\'evy measure.

We denote the $d \times d$ identity matrix by $I_d$ and the determinant of $\A$ by $\det \A$. Set $\A^* = \left(\det \A\right)^{1/d}I_d$ and define $X^*_t$ to be the  rotationally invariant L\'evy process  in $\bR^d$  associated to the triple $\left(0,
\A^*,\phi^*(y) dy \right)$. We will often refer to $X^*_t$ as the symmetrization of $X_t$. 

Notice that  
\begin{equation}    E^x\left[\, e^{i \xi \cdot X^*_t}\,\right]  = e^{-t
\Psi^*(\xi)+i \xi \cdot x},\label{levyX*} \end{equation}
where 
\begin{eqnarray*} 
\Psi^*(\xi)&=& \frac{1}{2}\langle \A^* \cdot \xi, \xi \rangle +
\int_{\bR^d}\left[\, 1 - e^{i\,\xi \cdot y} \,\right]
\phi^*(y)\,dy\\
&=&\frac{1}{2}\langle \A^* \cdot \xi, \xi \rangle +
\int_{\bR^d}\left[\, 1 - \cos(\xi \cdot y) \,\right]\phi^*(y)\,dy,
\end{eqnarray*}
where the last inequality follows from the fact that $\phi^*$ is symmetric and $y\to \sin(\xi \cdot y)$ is antisymmetric. 
 
The next two theorems are the main results of this paper. 

\begin{theorem}\label{fdd} 
Suppose  $X_t$ is a L\'evy process with  L\'evy measure absolutely continuous with respect to the Lebesgue measure and  let $X_t^{*}$ be the symmetrization of $X_t$ constructed as above. Let $f_1,\ldots,f_m$  be nonnegative continuous  functions and let $D_1,\ldots,D_m$  be domains in $\bR^d$. 
Then for all $z \in \bR^d$,

\begin{equation} \label{fdd0}
E^{z}\,\left[\prod_{i=1}^{m} f_i( X_{t_i})\, \bI_{D_i}( X_{t_i})\,\right]\\  
\leq E^0\,\left[\prod_{i=1}^{m}   f^*_i( X^*_{t_i})\,\bI_{D_i^*}( X^*_{t_i})\,\right],
\end{equation}
for all $0 \leq t_1 \leq \ldots \leq t_m$. 
\end{theorem}

One easily proves that this result is not valid when the functions $f_1,\ldots,f_m$  are not continuous. However, if we assume further that the distributions of $X_t$ and $X_t^*$ are absolutely continuous with respect to the Lebesgue measure, we can extend Theorem \ref{fdd} to measurable functions. 

\begin{theorem}\label{fdd1} 
Suppose  $X_t$ is a L\'evy process with  L\'evy measure absolutely continuous with respect to the Lebesgue measure and  let $X_t^{*}$ be the symmetrization of $X_t$ as constructed above. Assume further that for all $t>0$ the distributions of  $X_t$ and $X_t^{*}$  are absolutely continuous with respect to the Lebesgue measure. That is, for all $t>0$, 
$$
P^x\{X_t\in A\}=\int_A p(t, x, y) dy
$$
and 
$$
P^x\{X_t^*\in A\}=\int_A p^*(t, x, y) dy,
$$
for any Borel set $A\subset\bR^d$. 
Let $f_1,\ldots,f_m$, $m\geq 1$,  be nonnegative  measurable functions.  Then for all $z \in \bR^d$,

\begin{equation} \label{fdd4}
E^{z}\,\left[\prod_{i=1}^{m} f_i( X_{t_i})\right]\\  
\leq E^0\,\left[\prod_{i=1}^{m} f^*_i( X^*_{t_i})\right],\nonumber
\end{equation}
for all $0 \leq t_1 \leq \ldots \leq t_m$. 
\end{theorem}

\begin{remark}
A sufficient condition for the absolute continuity of the law of a L\'evy process is given in \cite{sato}, page 177. In our case this is satisfy by both $X_t$ and $X_t^*$ whenever $\det(\A)>0$ or  $\phi\not\in L^1(\bR^d)$.\end{remark}

As we shall see below, Theorem \ref{fdd} implies a generalization of  (\ref{exitbrownian})  to L\'evy processes whose L\'evy measure is absolutely continuous with respect to the Lebesgue measure.  In fact, we will obtain a more general result which applies to Schr\"odinger perturbations of L\'evy semigroups. Let $D\subset \bR^d$ be a domain of finite measure, and consider  

\[ \tau_{D}^X=\inf\left\{ t> 0: X_t \notin D\right\},\]
the first exit time of $X_t$ from $D$. We also have the corresponding quantity  $\tau_{D^*}^{X^*}$ for $X_t^{*}$ in $D^*$. As explained in \S5,  the following  isoperimetric--type inequality  is a consequence of Theorem \ref{fdd}.

\begin{theorem} \label{levy}
Let $D$ be a domain in $\bR^d$ of finite measure and  $f$ and $V$ be  nonnegative continuous functions. Suppose $X_t$ is a L\'evy process with  L\'evy measure absolutely continuous with respect to the Lebesgue measure and  $X_t^{*}$ is the symmetrization of $X_t$. Then  
for all $z \in \bR^d$ and all $t>0$, 
\begin{eqnarray}\label{main}
&&E^{z}\Big\{\, f(X_t)\,\exp\left( -\int_0^t V(X_s)ds\,\right);\,\tau_D^X>t \,\Big\}\\
&\leq & E^{0}\Big\{\, f^*(X^*_t)\,\exp\left( -\int_0^t V^*(X^*_s)ds\,\right);\,\tau_{D^*}^{X^*}>t \,\Big\}.\nonumber
\end{eqnarray}
\end{theorem}

Our symmetrization results  are based  on the following now classical rearrangement inequality  of Brascamp, Lieb and Luttinger \cite{brascamp}.

\begin{theorem}{\label{bll}}
Let  $f_1, \dots, f_m$ be nonnegative functions in  $\bR^{d}$ 
and denote by 
$f_1^*, \dots, f_m^*$ be their symmetric decreasing
rearrangements. Then 
\begin{eqnarray*} 
\int_{\bR^d}\ldots \int_{\bR^d} \,\prod_{j=1}^m f_{j}\left( \sum_{i=1}^{k} b_{ji} x_{i}\right) \,dx_{1}\cdots dx_{k}
&\leq&\\ \int_{\bR^d} \ldots \int_{\bR^d}\, \prod_{j=1}^m f_{j}^*\left(\sum_{i=1}^{k} b_{ji} x_{i}\right)\,dx_{1}\cdots dx_{k},
\end{eqnarray*}
for all positive integers $k,m$, and  any $m\times k$ matrix $B=[b_{ji}]$.
\end{theorem}


As explained in \cite{banuelos} and \cite{mendez}, if we additionally assume that the process $X_t$ is isotropic unimodal, Theorem \ref{fdd}
is an immediate consequence of  Theorem \ref{bll}.  Recall that  $X_t$ is isotropic unimodal if it has  transition densities $p(t, x,y)$ of the form
\begin{equation} \label{rotation} p(t,x,y)=q_t(|x-y|),\end{equation}
where $q_t$ is a function such that 
\[q_t(r_1) \leq q_t(r_2),\]
for all $r_1 \geq r_2$ and all $t>0$. Thus for such L\'evy processes (with $y$ fixed)
\[\left[\,p(t, \cdot, y)\,\right]^*=p(t, \cdot, 0),\] 
and  $X_t=X^*_t$. This class of L\'evy processes  includes the Brownian motion, rotational invariant symmetric $\alpha$-stable processes, relativistic stable processes and any other subordinations of the Brownian motion.  Notice that in our more general setting, and under the assumption that the distribution of  $X_t$ is absolutely continuos relative to the Lebesgue measure, we cannot even ensure that  
$\left[\,p(t, \cdot,y)\right]^*$ is the transition density of a L\'evy processes.



The rest of the paper is organized as follows. In \S2 we will prove Theorem \ref{fdd} for Compound Poisson processes. We will consider the case of  Gaussian L\'evy processes  in \S3.  Theorem \ref{fdd} and Theorem \ref{fdd1} are proved in \S4, using  a weak approximation of  $X_t$ and $X^*_t$ by 
L\'evy processes of the form $G_t+C_t$, where $G_t$ is a nondegenerate
Gaussian process and $C_t$ is an independent compound process. We will then show  some of the applications  in \S5.  For the convenience of the reader, and for completeness, we include an appendix in \S6 with various facts on symmetrization used in the proofs.

\section{Symmetrization of compound Poisson processes}

In this section we prove a version of the inequality (\ref{fdd0}) for compound Poisson processes, in the case that $D_i=\bR^d$ for all 
$1 \leq i \leq m$.  This result, combined with the results in \S3,  will lead to a proof of Theorem \ref{fdd}. 

We start by recalling the structure of compound Poisson processes in terms of random walks. If $C_t$ is a  compound Poisson process, starting at $x$, then its  characteristic
function is given by 
\begin{equation}  E^x\left(\, e^{i \xi \cdot C_t}\, \right)  = e^{i x \cdot \xi -t
\Psi_C(\xi)}, \label{levyC} 
\end{equation} 
where 
\begin{equation} \Psi_C(\xi)= c
\int_{\bR^d}\,\left[\, 1-e^{i \xi \cdot y}\, \right] \,
\phi(y)\, dy, \nonumber
\end{equation}
and  $\phi$ is a probability density.  We now use the fact that $C_t$ can be written in terms of  sums of independent random variables.  That is, by
Theorem 4.3 \cite{sato} there exist
a Poisson process $N_t$ with parameter $c >0$, and a sequence of i.i.d. random variables
$\{X_n\}_{n=1}^{\infty}$ such that 

\begin{enumerate}
\item  $\{N_t\}_{ t >0}$ and $\{X_n\}_{n=1}^{\infty}$ are independent,
\item  $\phi(y)$ is the density of the distribution of $X_i$, $i \geq 1$,
\item  $C_t= S_{N_t}+x$, where $S_n=X_1+\ldots+X_n$ and $S_0=0$.
\end{enumerate}
Hence if $f$ is a nonnegative Borel function, then

\begin{eqnarray}
\nonumber E^x\left[\, f\left(C_t\right) \,\right] & = & E^x\left[ \,f\left(S_{N_t}\right) \,
\right]\\ \label{dist-ct}  &=& \sum_{n=0}^{\infty} P\left[\,
N_t=n\,\right]\, E\left[ \,f\left(x+S_{n}\right)\right].
\end{eqnarray}

Let  $ \phi^*$ be the symmetric decreasing rearrangement of $\phi$. Since  
\[ \int_{\bR^d} \, \phi^*(y)\, dy =  \int_{\bR^d} \, \phi(y) \, dy =1,\]
we can consider a new sequence of i.i.d. random variables $\{X^*_n\}_{n=1}^{\infty}$ independent of $N_t$ such that   $ \phi^*(y)$ is the
density of $X^*_n$. Define  $S^*_n=X^*_1+\ldots+X^*_n$ to be the corresponding random walk and   $C^*_t$  the compound Poisson process given by
\[ C^*_t=S^*_{N_{t}}.\]

Notice that the distribution $\mu_t$  of $C_t$ is not absolutely continuous with respect to  Lebesgue measure. However, if $C_0=x$  we have the following representation 

\begin{equation} \label{densityC}
\mu_t= P\left[\,N_t=0\,\right] \delta_{x} 
+ \sum_{k=1}^{\infty} P\left[\,N_t=k\,\right]\,\mu_{k}(x),
\end{equation}
with $\mu_k$ the distribution of $S_k$.  That is, 
\begin{eqnarray*}
E^{x}\left[\,f\left(\,S_k\,\right)\,\right]&=&\int_{\bR^d} f(x+y) d\mu_k(y)\\
&=& \int_{\bR^d} \ldots \int_{\bR^d} \,f\left(\,\sum_{j=0}^k x_j\,\right) \,\prod_{i=1}^k \phi(x_i)\, dx_1\ldots dx_{k}.
\end{eqnarray*}
Thus if $f$ is a bounded measurable function we have that
\[ f^*(S^*_0)=f^*(0)= \|f\|_{L^\infty}\] 
and the inequality 
\[ f\left(\,S_0+x\,\right)=f(x) \leq f^*(S^*_0),\]
can only be asserted to hold almost everywhere.

The next result is a version of inequality (\ref{fdd0}) for random walks where the functions are only assume to be measurable but the conclusion is only a.e. with respect to the Lebesgue measure. We label it as  ``Theorem" because it may be of some independent interest. 
\begin{theorem} {\label{sn}}
Let $ f_1,\ldots,f_m$ nonnegative functions  and $k_1
\leq \ldots \leq k_m$ nonnegative integers.   Then 
\begin{equation} \label{fddrw}
 E\left[ \,\prod_{i=1}^{m} f_i( x_0+S_{k_i})   \,\right]
\leq
 E\left[ \,\prod_{i=1}^{m} f^*_i( S^*
 _{k_i})   \,\right], 
\end{equation}
almost everywhere in $x_0$, with respect to Lebesgue measure. 
In the case that $ f_1,\ldots,f_m$  are continuous, (\ref{fddrw}) holds pointwise. 
\end{theorem}
\begin{proof}
Given that  $X_1, \ldots, X_{k_m}$ are i.i.d we can apply  Theorem \ref{bll}  to obtain   that 

\begin{eqnarray} \nonumber
&&E\left[ \,\prod_{i=1}^{m} f_i( x_0+S_{k_i})   \,\right]\\\nonumber
&=&  E\left[ \,\prod_{i=1}^{m} f_i( x_0+X_1+...+X_{k_i})   \,\right]\\\nonumber
&=&\int_{\bR^d}\ldots \int_{\bR^d} 
\,\left[\prod_{i=1}^m  \,f_i\left(\sum_{j=0}^{k_i} x_j\,\right)\right]\,
\prod_{i=1}^{k_m}  \phi(x_i)\, dx_1\ldots dx_m\\\nonumber
&\leq&\int_{\bR^d}\ldots \int_{\bR^d} 
\,\left[\prod_{i=1}^m  \,f^*_i\left(\sum_{j=1}^{k_i} x_j\right)\right]\,
\prod_{i=1}^{k_m}\phi^*(x_i)\, dx_1\ldots dx_m\\
&=& E\left[ \,\prod_{i=1}^{m} f^*_i(S^*_{k_i})   \,\right].
\end{eqnarray}
\end{proof}

We can now prove the inequality  (\ref{fdd0}) for the compound Poisson process $C_t$ under the assumption that all the domains are $\bR^d$. Let $ f_1,\ldots,f_m$ be nonnegative continuous functions. Since $N_t$ is independent of $S_{k}$ and $S^*_{k}$, we can combine (\ref{dist-ct}) and Theorem \ref{sn} to obtain
\begin{eqnarray}\label{fddct}
&&E^x\left[ \,\prod_{i=1}^{m} f_i( S_{N_{t_i}})\,\right] \nonumber  \\
&=&  \sum_{k_1\leq k_2\leq \ldots \leq k_m}^{\infty}\nonumber
P\left[N_{t_1}=k_1,\ldots,N_{t_m}=k_m\right]\,E\left[\,\prod_{i=1}^{m} f_i(x+ S_{k_i}) \right]\nonumber \\
&\leq&  \sum_{k_1\leq k_2\leq \ldots \leq k_m}^{\infty} P\left[ N_{t_1}=k_1,\ldots,N_{t_m}=k_m \right] E\left[ \,\prod_{i=1}^{m} f^*_i( S^*_{k_i})   \, \right]\nonumber\\
&=&  E^0\left[ \,\prod_{i=1}^{m} f^*_i( S^*_{N_{t_i}})\,\right]. 
\end{eqnarray} 
Thus 
\begin{equation}
E^x\left[ \,\prod_{i=1}^{m} f_i( C_{t_i})\,\right] \leq E^0\left[ \,\prod_{i=1}^{m} f^*_i( C^*_{t_i})\,\right],
\end{equation}
which is desired result.

\section{Symmetrization of  Gaussian  processes}

Let $G_t$  be a  nondegenerate Gaussian process. Then there
exist $b \in \bR^d$ and   a strictly positive definite symmetric
$d \times d$ matrix $\A$ such that the density of $G_t$ is given by
\[ f_{\A,b}(t,x)= \frac{1}{ \left[\,2 t \pi\,\right]^{d/2} \sqrt{ \det  \A}}\,
\exp\left[ \, -\frac{1}{2t}\,\left\langle \, (x-t b), \A^{-1}\cdot (x-t b)\,\right \rangle
\right],\] for all $x \in \bR^d$ and all $t>0$.

Let us first assume that  $b=0$.  Let $ u > 0$, then 
\begin{eqnarray*}
&& \left\{ x \in \bR^d: f_{\A,0}(t,x) >u\right\}\\ &=&
 \left\{ x \in
\bR^d: \langle\, x, \A^{-1} \cdot x\,\rangle\, < t\, \ln\left[
\frac{1}{(2 t \pi)^d u^2 \det \A} \right] \right\}\\
&=& \left\{ x
\in \bR^d: \langle\, \A^{-1/2} \cdot x, \A^{-1/2} \cdot  x\,\rangle\,
< \,t\, \ln\left[ \frac{1}{(2 t \pi)^d u^2 \det \A} \right] \right\}.\nonumber 
\end{eqnarray*}
A change of variables implies that

\begin{equation}\nonumber
m \left\{ x \in \bR^d: f_{\A,0}(t,x) >u\right\} = \frac{1}{ \left[
\det \A \right]^{1/2}} m \left\{\, B( r_{\A,d,u,t})\,\right\},
\end{equation}
where 
$$ r_{\A,d,u,t}= \,t\, \ln\left[ \frac{1}{(2 t \pi)^d u^2 \det \A}
\right].$$

Consider the diagonal matrix
\[\A^*=\,\left( \,\det \A\,\right)^{\frac{1}{d}} I_d.\]
Then
\begin{equation}\nonumber
m \left\{ x \in \bR^d: f_{\A,0}(t,x) >u\right\} =m \left\{ x \in
\bR^d: f_{\A^*,0}(t,x) >u\right\},
\end{equation}
for all $ u > 0 $.
Given that $f_{\A^*,0}(t,x)$ is rotational invariant and radially decreasing, we conclude that  

\begin{equation}\label{gaussian}
\left[ \, f_{\A,b}(t,x) \,\right]^* = \left[ \, f_{\A,0}(t,x-tb)
\,\right]^* =  \, f_{\A^*,0}(t,x).\\
\end{equation}

If $G_t$  is a   degenerate Gaussian process, then 
\begin{equation} 
E\left(\, e^{i \xi \cdot G_t}\, \right)  = \exp\left({i t b\cdot \xi- i\frac{t}{2}\langle \A \cdot \xi, \xi \rangle}\right),
\end{equation}
where $\A$ is a positive definite $d \times d$  matrix such that $\det \A=0$.

Let $\{v_1,\ldots,v_d\}$ be the orthonormal eigenvectors of $\A$ with eigenvalues $\lambda_1, \ldots, \lambda_d$.  We can assume that $\{\lambda_1,\ldots,\lambda_k\}$, $1 \leq k <d$, are   the nonzero eigenvalues  of $\A$. 
Let  $W$  be the subspace spanned  by $v_1,\ldots,v_k$. Then $G_t$ can be identified with a non degenerate Gaussian process in the lower dimension space $W$ and 
\[P^z\left[ \,G_t \in D\,\right]= P^z\left[ \,G_t \in P_W(D)\,\right],\]
where $P_W(D)$ is the projection of $D$ on the space $W$. 

Define  $\A^*$ to be the symmetric positive defined matrix with eigenvectors $ v_1,\ldots,v_d$ such that
\[\A^* v_i= 0, \, k< i \leq d,\]
and
\[\A^* v_i= \lambda v_i,\, 1 \leq i \leq k,\]
where
\[\lambda= \left( \lambda_1\cdots\lambda_k\right)^{1/k}.\]
The arguments of this section imply that 
\begin{eqnarray*}
P^z\left[ \,G_t \in D\,\right]&=& P^z\left[ \,G_t \in P_W(D)\,\right]\\
                              &=&  P^0\left[ \,G^*_t \in D_W^{*} \,\right].
\end{eqnarray*}
where $D_W^{*}$ is the ball in $W$, centered at the origin, with the same $k$-dimension measure as  $P_W(D)$. 
Hence the corresponding symmetrization  for this processes should be done in lower dimensions. 
 
\section{Symmetrization of L\'evy processes: Proof of Theorem \ref{fdd}}

We will now consider  general L\'evy processes whose L\'evy measures are absolutely continuous with respect to the Lebesgue measure. 
Our proof requires  two basic results on symmetrization of functions that are included in the Appendix in \S6.

Recall that under our assumptions 
\begin{equation} \nonumber   E^x\left[\, e^{i \xi \cdot X_t}\,\right]  = e^{-t
\Psi(\xi)+i \xi \cdot x}, \end{equation}
where
\[ \Psi(\xi)= -i \langle b, \xi \rangle + \frac{1}{2}\langle \A \cdot \xi, \xi \rangle +
\int_{\bR^d}\left[\, 1 + i \langle \xi,y \rangle \bI_{B} - e^{i\,\xi \cdot y} \,\right]\,\phi(y)\,
dy, \] 
$B$ is the unit ball centered at the origin and  $\phi$ is such that 
\begin{equation}\int_{\bR^d} \frac{|y|^2}{1+|y|^2}\,\phi(y)\, dy < \infty \,\label{phi}.\end{equation}

Consider the sequence 
\[\phi_n(y)= \phi(y)\,\bI_{\{ t \in \bR: \frac{1}{n} < t \} }(|y|),\] 
and let $\phi_n^*(y)$ be its symmetric decreasing rearrangement. 
Thanks to (\ref{phi}), 
\[c_n= \int_{\bR^d} \,\phi_n(y)\, dy <  \infty \,,\]
and
\[ \int_{B}\, |y_i|\,\phi_n(y) \,dy  < \infty , \, 1\leq i \leq d,\] 
where again $B$ is the unit ball.

Consider $C_{n,t}$ a compound Poisson process  with characteristic function 
\begin{equation}  E\left(\, e^{i \xi \cdot C_{n,t}}\, \right)  = e^{-t
\Psi_{C,n}(\xi)}, \end{equation} 
where 
\begin{equation} \Psi_{C,n}(\xi)= c_n
\int_{\bR^d}\,\left[\, 1-e^{i \xi \cdot y}\, \right] \,
\frac{\phi_n(y)}{c_n}\, dy. \nonumber
\end{equation}

Given that all the eigenvalues of $\A$ are nonnegative,  if $\{\epsilon_n \}_{n=1}^\infty$ is a sequence of positive numbers converging to zero, then  $\A_n=\A+\epsilon_n I_d$ is a sequence of nonnegative  nonsingular  matrices. Let  $G_{n,t}$ be a Gaussian process starting at $x$, independent of $C_{n, t}$, and  associated with the matrix  $\A_n$  and the vector
 $ b_n=  b - \int_{B} y \,\phi_n(y)\, dy.$
Set 
$X_{n,t}=C_{n,t}+ G_{n,t}.$
Since  $C_{n,t}$ and  $G_{n,t}$ are independent, 
\begin{equation} \nonumber   E^x\left[\, e^{i \xi \cdot X_{n,t}}\,\right]  = e^{-t
\Psi_n(\xi)+i \xi \cdot x}, \end{equation}
where
\begin{eqnarray}
\Psi_n(\xi) &=& -i \langle b_n, \xi \rangle  + \frac{1}{2}\langle \A_n \cdot \xi, \xi \rangle +
\int_{\bR^d}\left[\, 1  - e^{i\,\xi \cdot y} \,\right]\,\phi_n(y)\,
dy\\ \nonumber
&=& -i \langle b, \xi \rangle + \frac{1}{2}\langle \A_n \cdot \xi, \xi \rangle +
\int_{\bR^d}\left[\, 1 + i \langle \xi,y \rangle \bI_{B} - e^{i\,\xi \cdot y} \,\right]\,\phi_n(y)\,
dy. \end{eqnarray}

Let $S_{n,k}=X^n_1+\ldots+X^n_k$ be the random walk associated to $C_{n,t}$. If $f_1,\ldots,f_m$ are nonnegative continuous functions and $t_1\leq\ldots\leq t_m$, then
\begin{eqnarray}\label{fddn}
&&E^x\left[\,\prod_{i=1}^{m} f_i( X_{n,t_i})\,\right]\\\nonumber
&=& E^x\left[\,\prod_{i=1}^{m} f_i\left(\, C_{n,t_i}+G_{n,t_i}\,\right)\,\right]\\
&=&\sum_{k_1\leq k_2\leq \ldots \leq k_m}^{\infty}
P\left[N_{t_1}=k_1,\ldots,N_{t_m}=k_m\right]\,E^x\left[\,\prod_{i=1}^{m} f_i\left(\, S_{n,k_i}+G_{n,t_i}\,\right)\,\right].\nonumber 
\end{eqnarray}
Now  Theorem \ref{bll} and equality  (\ref{gaussian}) imply that 
\begin{eqnarray}
&& E^x\left[\,\prod_{i=1}^{m} f_i\left(\,G_{n,t_i}+ \sum_{j=1}^{k_i}X^n_j \,\right)\,\right]\\ \nonumber
&=& \int_{\bR^d}\ldots\int_{\bR^d} \,\prod_{i=1}^{m} f_i\left(\sum_{j=0}^{k_i}x_j \,\right) \,f_{\A_n,b_n}(t,x_0-x)\,\prod_{j=1}^{k_m} \phi(x_j)\,dx_0\ldots dx_{k_m}\\ \nonumber
&\leq& \int_{\bR^d}\ldots\int_{\bR^d} \,\prod_{i=1}^{m} f^*_i\left(\,\sum_{j=0}^{k_i}x_j \,\right) \,f^*_{\A_n,0}(t,x_0)\,\prod_{j=1}^{k_m} \phi^*(x_j)\,dx_0\ldots dx_{k_m}\\ \nonumber
&=& \int_{\bR^d}\ldots\int_{\bR^d} \,\prod_{i=1}^{m} f^*_i\left(\,\sum_{j=0}^{k_i}x_j \,\right) \,f_{\A^*_n,0}(t,x_0)\,\prod_{j=1}^{k_m} \phi^*(x_j)\,dx_0\ldots dx_{k_m}\\ \nonumber
&=&E^x\left[\,\prod_{i=1}^{m} f^*_i\left(\,G^*_{n,t_i}+ S^*_{n,k} \,\right)\,\right].
\end{eqnarray}
This implies that 

\begin{equation}\label{Xn}
E^x\left[\,\prod_{i=1}^{m} f_i( X_{n,t_i})\,
\right] \leq E^0\left[ \,\prod_{i=1}^{m} f^*_i( X
^*_{n,t_i})\,\right].
\end{equation}

Theorem \ref{fdd} will be  a consequence of (\ref{Xn}) and the following result on weak convergence.

\begin{theorem}\label{weak}
Let  $f_1,\ldots,f_k$ be nonnegative bounded continuous functions,  and $0 <t_1<
\ldots < t_m$.  Then for all $ x \in \bR^d$,
\begin{equation}
\lim_{n \to \infty} E^x\left[\,\prod_{i=1}^{k} f_i( X_{n,t_i})\,
\right] = E^x\left[\,\prod_{i=1}^{k} f_i( X_{t_i})\,\right],
\end{equation}
and  
\begin{eqnarray}\nonumber
\lim_{n \to \infty} E^x\left[ \,\prod_{i=1}^{k} f_i( X
^*_{n,t_i})\,\right] = E^x\left[\,\prod_{i=1}^{k} f_i( X^*_{t_i})\,\right].
\end{eqnarray}
\end{theorem}
\begin{proof}

Notice that for all $ \xi \in  \bR^d$, 
\[ \lim_{n\to \infty} \langle \A_n \cdot \xi, \xi \rangle=\langle \A \cdot \xi, \xi \rangle.\]
Given that there exists $C \in \bR^{+}$ such that, 
\begin{equation}
\left|\, 1 + i \langle \xi , y \rangle  - e^{i\,\xi \cdot y} \,\right|\,\phi_n(y)\, \leq C\,|\xi|^2\,
\, |y|^2 \phi(y)\,
 < \infty,
\end{equation}
for all $y \in B$, and
\begin{equation}
\left|\, 1  - e^{i\,\xi \cdot y} \,\right|\,\phi_n(y)\, \leq 2\, \phi(y)\, < \infty,
\end{equation}
for all $ y \in \bR^d \setminus B$, 
it follows from the Dominated Convergence Theorem that
 
\begin{eqnarray}
&& \lim_{n\to \infty} \Psi_n(\xi) \\ \nonumber
&=& \lim_{n\to \infty} \left(\,-i \langle b, \xi \rangle + \frac{1}{2}\langle \A_n \cdot \xi, \xi \rangle +
\int_{\bR^d}\left[\, 1 + i \langle \xi,y \rangle \bI_{B} - e^{i\,\xi \cdot y} \,\right]\,\phi_n(y)\,
dy\right)\\
&=& \left(\,-i \langle b, \xi \rangle + \frac{1}{2}\langle \A \cdot \xi, \xi \rangle +
\int_{\bR^d}\left[\, 1 + i \langle \xi,y \rangle \bI_{B} - e^{i\,\xi \cdot y} \,\right]\,\phi(y)\,
dy\right).\nonumber
\end{eqnarray}

We conclude that 
\begin{equation}  \label{weakonedimensional}
\lim_{n \to \infty}   E^x\left[\, e^{i \xi \cdot X_{n,t}}\,\right] =
E^x\left[\, e^{i \xi \cdot X_{t}}\,\right].
\end{equation}

On the other hand, using the last two inequalities and the  fact that
\[ \lim_{n\to \infty} \det \A_n = \det \A,\]
we can easily prove that 
\[ \lim_{n\to \infty} \langle \A^*_n \cdot \xi, \xi \rangle=\langle \A^* \cdot \xi, \xi \rangle.\]
Lemma \ref{order} implies that 
\[ \phi^*_n(x) \leq \phi^*(x), \]
for all $x \in \bR$, and all $n \geq 1$. 
In addition, Proposition \ref{limit-th} gives that 
\[\lim_{n \to \infty} \phi^*_{n}=\phi^*, \text{ a.e. }.\]
Thus the same argument used to prove (\ref{weakonedimensional}) yields
\begin{equation}  \label{weakonedimensional*}
\lim_{n \to \infty}   E^x\left[\, e^{i \xi \cdot X_{n,t}^*}\,\right] = E^x\left[\, e^{i \xi \cdot X_{t}^*}\,\right].
\end{equation}
Now, if $\xi_1, \ldots, \xi_m \in \bR^d$, then 
\begin{eqnarray}  \nonumber
\sum_{j=1}^m \xi_j \cdot X_{n,t_j}& =&\left( \xi_1+\ldots+ \xi_m\right) \cdot X_{n,t_1}\\&+&
\sum_{j=2}^{m}  \,( \xi_m+\ldots+\xi_{j})\, \cdot \left(\,X_{n,t_{j}}-X_{n,t_{j-1}}\,\right). 
\end{eqnarray}
Since  $t_1<\ldots < t_m$ we have that 
\begin{eqnarray}  
&&E^x\left\{\,\exp\left[i \,\sum_{j=1}^m \xi_j \cdot X_{n,t_j}\,\right]\,\right\}\\\nonumber
&=&E^x\left\{\,\exp\left[\, i\left(\,\xi_1+\ldots+ \xi_m\,\right) \cdot X_{n,t_1}\right]\,\right\}\\&\times&\nonumber
\prod_{j=2}^{m}  E^0\left\{\,\exp\left[\,i( \xi_m+\ldots+\xi_{j})\, \cdot \left(\,X_{n,t_{j}-t_{j-1}}\,\right)\,\right]\,\right\}.
\end{eqnarray}
The desired result  immediately follows from (\ref{weakonedimensional}), (\ref{weakonedimensional*}) and the fact that our characteristic functions are continuous at $0$. This last observation follows from the L\'evy-Khintchine formula.
\end{proof}

Combining (\ref{Xn}) and Theorem \ref{weak},  we obtain 
\begin{equation}\label{bounded}
E^x\left[\,\prod_{i=1}^{m} f_i( X_{t_i})\,
\right] \leq E^0\left[ \,\prod_{i=1}^{m} f^*_i( X
^*_{t_i})\,\right].
\end{equation}
for all  nonnegative bounded continuous functions  $f_1,\dots,f_m$.

Let  $f$ be a nonnegative  continuous functions, and  consider the sequence  
\[ f_{n}= \max\{ f,n\}.\]
Then 
\[ 0 \leq f_{n}(x) \leq f_{n+1}(x) \leq f(x), \text{ for all }  x \in \bR^d.\] 
Thus Proposition \ref{order} implies that 
\[ 0 \leq f^*_{n}(x) \leq f^*_{n+1}(x) , \text{ for all }  x \in \bR^d.\] 
Since $f^*$ is continuous, Proposition \ref{limit-th} yield 
\[ \lim_{n \to \infty} f^*_{n}(x)=f^*(x) , \text{ for all }  x \in \bR^d.\] 
The Monotone  Convergence Theorem for the laws of $X_t$ and $X_t^*$ imply (\ref{bounded}) for all nonnegative continuous functions.

To finish the proof of Theorem \ref{fdd}, we must show that we may replace a continuous function $f_i$ by the indicator (characteristic) function of a domain of finite volume.  
Let $O$ be a open set of finite volume and consider
\[ \psi_n\left(\,x\,\right)= 1- \left(\, 1-n d(x,F)\,\right)_{+}, \text{ where } F=\bR^d\setminus O.\] 
Notice that  $\psi_n(x)=0$  if $ x \in F$, and  $\psi_n(x)=1$, if  $d(x,\bR^d\setminus O) \geq \frac{1}{n}$.
In addition, 
\[ 0 \leq \psi_n(x) \leq \psi_{n+1}(x) \leq 1, \text{ for all }  x \in \bR^d.\] 
By Proposition \ref{order},  
\[ 0 \leq \psi^*_n(x) \leq \psi^*_{n+1}(x) \leq 1, \text{ for all }  x \in \bR^d.\] 
Therefore
\[ \psi_n^*(x)= \int_0^\infty  \bI_{\left\{ \psi_n^* > t\right\}}(x)\, dt=\int_0^1  \bI_{\left\{ \psi_n > t\right\}^*}(x) \,dt.\]

Let $O_n= \left\{ x : d(x,\bR^d\setminus O) > \frac{1}{n}\right\}$, $O^*=B(0,r)$, and  $O_n^*=B(0,r_n)$. If $ 0< t < 1$, then 

\[ m\left\{ \psi_n > t\right\}^*=m\left\{ \psi_n > t\right\} > m\left\{ O_n^*\right\}.\]
Hence, for all $x$,  
\[\bI_{ B(0,r_n)}(x) <  \bI_{\left\{ \psi_n > t\right\}^*}(x) < \bI_{ B(0,r)}(x).\]
Integrating in $t$ we obtain 
\[ \bI_{ B(0,r_n)}(x) \leq \psi_n^*(x) \leq \bI_{ B(0,r)}(x).\]
We conclude that 
\[\lim_{n \to \infty}  \psi_n^*(x)= \bI_{ B(0,r)}(x), \text{ for all } x \in \bR,\]
and Theorem \ref{fdd} follows from the  Monotone Convergence Theorem. \qed
\vskip.5cm

Now we will prove Theorem \ref{fdd1}. Without lost of generality we can assume that the functions 
$f_1,\ldots,f_m$ are finite almost everywhere with respect to the Lebesgue measure.
Let $ 1 \leq i \leq m$. We first assume that  there exists a constant $M_i$
\[ f_i(x) \leq M_i, \]
for all $x \in \bR$.  Then there exists a sequence of nonnegative continuous functions  $\{\phi_n\}_{n=1}^\infty$ such that
\[ \lim_{n \to \infty} \phi_n= f_i,\] 
almost everywhere with respect to the Lebesgue measure, and 
\[ \phi_n(x) \leq M_i, \]
for all $x \in \bR$.
By proposition 6.2  
\[ \lim_{n \to \infty} \phi^*_n= f_i^*,\]
almost everywhere with respect to the Lebesgue measure. Thus 
the absolute continuity of the laws of $X_{t_i}$ and $X_{t_i}^*$ with respect to the Lebesgue measure and the Dominated Convergence Theorem yield (\ref{bounded}) if  the functions 
$f_1,\ldots,f_m$ are bounded.

Finally, if   $f_i$ is not bounded, consider the sequence 
\[ f_n= \max\{ f_i ,n\}.\]
As before, Proposition \ref{order} and Proposition \ref{limit-th} imply 

\[  f^*_n \leq  f^*_{n+1}\leq f_i^*\text{ for all } x \in \bR,\]
and 
\[\lim_{n \to \infty} f^*_n = f_i^*,\]
almost everywhere with respect to the Lebesgue measure. As before, Theorem \ref{fdd1} 
follows from the absolute continuity of the laws of $X_{t_i}$ and $X_{t_i}^*$ with respect to the Lebesgue measure and the Monotone Convergence Theorem.

\section{Some Applications}

In this section we give several applications of Theorem \ref{fdd}, we begin with the proof of Theorem \ref{levy}. Recall that 
\[ \tau_{D}^X=\inf\left\{ t>0: X_t \notin D\right\}\] 
is the first exit time of $X_t$ from a domain $D$. Let $D_k$ be a sequence of bounded domains with smooth boundaries such that $ \overline{D_k} \subset D_{k+1}$, and $\cup_{k=1}^{\infty}D_k=D$.  Since any L\'evy process has a version  with right continuous paths, we have 
\begin{eqnarray}\label{0fdd}
&&E^{z_0}\left\{\, f(X_t)\,\exp\left( -\int_0^t V(X_s)ds\,\right);\, \tau^X_D>t \,\right\}\\ \nonumber
&=&E^{z_0}\left\{ \,f(X_t)\,\exp\left( -\int_0^t V(X_s)ds\,\right);\, X_{s} \in D, \,\forall s \in  [0, t]\,\right\}\\ \nonumber
&=&\lim_{m \to \infty}\lim_{k \to \infty} E^{z_0}
\left\{\,f(X_t)\,\exp\left(-\frac{t}{m}\sum_{i=1}^{m}V(X_{\frac{it}{m}})\,\right)
;\,X_{\frac{it}{m}} \in D_k,\,i= 1,\ldots,m \,
\right\}\\ \nonumber
&=&\lim_{m \to \infty}\lim_{k \to \infty} E^{z_0}
\left\{\,f(X_t)\,\prod_{i=1}^{m} \exp\left(-\frac{t}{m} V(X_{\frac{it}{m}})\,\right)\bI_{D_k}\left(X_{\frac{it}{m}}\right)
\right\}.
\end{eqnarray}
Since 
\[\left[\, \exp\left(-s V(x)\,\right)\,\right]^*= \exp\left(-s V^*(x)\,\right),\]
for all $s>0$ and all $ x \in \bR^d$, Theorem \ref{fdd} implies that 
\begin{eqnarray*}
&&E^{z_0}\left\{\,f(X_t)\,\prod_{i=1}^{m} \exp\left(-\frac{t}{m} V(X_{\frac{it}{m}})\,\right)\bI_{D_k}(X_{\frac{it}{m}})\right\}\\
&\leq&  E^{0}\left\{\,f^*(X^*_t)\,\prod_{i=1}^{m} \exp\left(-\frac{t}{m} V^*(X^*_{\frac{it}{m}})\,\right)\bI_{D^*_k}(X^*_{\frac{it}{m}})\right\}. \nonumber
\end{eqnarray*}
Hence we have the following 
\begin{eqnarray}\label{exit1}
&&E^{z}\Big\{\, f(X_t)\,\exp\left( -\int_0^t V(X_s)ds\,\right);\,\tau_D^X>t \,\Big\}\\
&\leq & E^{0}\Big\{\, f^*(X^*_t)\,\exp\left( -\int_0^t V^*(X^*_s)ds\,\right);\,\tau_{D^*}^{X^*}>t \,\Big\},\nonumber
\end{eqnarray}
which is Theorem (\ref{levy}).  Taking $V=0$ and $f=1$, gives 
\begin{equation}
P^{z}\Big\{\, \tau_D^X>t \,\Big\}
 \leq P^{0}\Big\{\,\tau_{D^*}^{X^*}>t \,\Big\}, 
\end{equation} 
which is a generalization of  inequality (\ref{exitbrownian}). 
Integrating this inequality with respect to $t$ gives the following result. 
\begin{corollary}
If $\psi$ is a nonnegative increasing
function,  then 
\begin{equation} E^z\left[\,\psi\left(\,\tau_{D}^X\,\right)\,\right]\,\leq
E^0\left[\,\psi\left(\,\tau_{D^*}^{X^*}   \right)\,\right],\end{equation}   
 for all  $ z \in D$. 
 In particular
\begin{equation} E^z\left[\,\left(\,\tau_{D}^X\,\right)^p\,\right]\,\leq
E^0\left[\,\left(\,\tau_{D^*}^{X^*}\,\right)^p\,\right],\end{equation} for all $0<p<\infty$.
\end{corollary}

Our results imply many isoperimetric inequalities for the potentials and the eigenvalues  of Schr\"odinger operators of the form
\[H^X_{D,V}=H^X_D + V,\]
where $H^X_D$ is the pseudo differential operator associated to $X_t$ with Dirichlet Boundary conditions on $D$. For the convenience of the reader we will give a brief description of the operators and semigroups associated to L\'evy processes.

For purposes of our formulae below we define the  Fourier transform of an $L^2(\R^d)$ function as 
\[\widehat{f}(\xi)=\frac{1}{(2\pi)^{d/2}}\int_{\bR^d}e^{-ix\cdot \xi} f(x) \,dx,\]
with 
\[f(x)=\frac{1}{(2\pi)^{d/2}}\int_{\bR^d}e^{ix\cdot \xi}\, \widehat{f}(\xi)\, d\xi.\]
We define the semigroup associated to the  L\'evy process $X_t$  by 
\begin{eqnarray*}
T_t f(x)&=&E^x [\,f(X_t)\,]\\
&=&\frac{1}{(2\pi)^{d/2}}\,E^0 \left[\,\int_{\bR^d}e^{i(X_t+x)\cdot\xi}\,\widehat{f}(\xi)\,d\xi\,\right]\\
&=&\frac{1}{(2\pi)^{d/2}} \int_{\bR^d}\,e^{ix\cdot\xi}\, E^0 \left[\,e^{iX_t\cdot\xi}\,\right]\widehat{f}(\xi)\,d\xi \\
&=&\frac{1}{(2\pi)^{d/2}}\int_{\bR^d}\,e^{ix\cdot\xi}\,e^{-t\Psi(\xi)}\,\widehat{f}(\xi)\,d\xi.
\end{eqnarray*}
This  semigroup takes $C_0(\bR^d)$ into itself. That is, it is a Feller semigroup.  From this we see that, at least formally for 
$f\in \mathcal{S}(\bR^d)$, the infinitesimal generator is 
\begin{eqnarray*}
H^X f(x)&=&-\frac{\partial T_t f(x)}{\partial t}\Big|_{t=0}
=\frac{1}{(2\pi)^{d/2}}\int_{\bR^d}\,
e^{ix\cdot\xi}\,\Psi(\xi)\,\hat f(\xi)\,d\xi. 
\end {eqnarray*}
Then the L\'evy-Khintchine  formula implies that the operator associated to $X_t$ is given by 
\begin{eqnarray}\label{generator}
H^X f(x)&=& \sum_{j=1}^{d} b_j \partial_jf(x) - \frac{1}{2}\sum_{j, k=1}^d a_{j k} \partial_j\partial_k f(x)\\
&+& \int_{\bR^d} \Big[\,f(x+y)-f(x)-y\cdot\nabla f(x)\, \bI_{\{|y|<1\}}\,\Big] d\nu(y),\nonumber
\end{eqnarray}
where $a_{jk}$ are the entries of the matrix $\A$.  
For instance:
\begin{enumerate}
\item If $X_t$ is a standard Brownian motion: 
$$H^X f=-\frac{1}{2}\Delta f.$$
\item If $X_t$ is a symmetric stable processes of order $0<\alpha<2$:  
$$H^X f=-\left(-\frac{1}{2}\Delta\right)^{\alpha/2}f.$$
\item If $X_t$ is a Poisson process of intensity $c$: $$H^X f(x)=c\Big[f(x+1)-f(x)\Big].$$
\item If $X_t$ is a compound Poisson process with measure $\nu$ and $c=1$:  
\begin{equation*}
H^X f(x) = \int \,\left[\,f(x+y)-f(x)\,\right]\,d\nu(y).
\end{equation*}
\end{enumerate}

In this paper we are interested not on the ``free" semigroup for $X_t$ but rather on its ``killed" semigroup and its perturbation by the potential $V$.  That is, we want properties of the semigroup 

\begin{equation} T_{t}^{D,V}\,f(z) = E^{z}\Big\{\, f(X_t)\,\exp\left( -\int_0^t V(X_s)ds\,\right);\,\tau_D^X>t \,\Big\},\end{equation}
defined for $t>0$, $ z \in D$,  and $f\in L^2(D)$. Recall our assumption that $V$ is nonnegative and continuous.  Thus,
\begin{eqnarray}\label{Vdomination}
|T_{t}^{D,V}\,f(z)|& =& \Big|E^{z}\Big\{\, f(X_t)\,\exp\left( -\int_0^t V(X_s)ds\,\right);\,\tau_D^X>t \,\Big\}\Big| \nonumber\\
&\leq &E^{z}\Big\{\, |f(X_t)|\,;\,\tau_D^X>t \,\Big\}=T_{t}^{D}|f|(z).
\end{eqnarray}

For the rest of the paper we shall assume  that the 
distributions of $X_t$ and $X_t^*$ have densities 
$p^X(t, z,w)$ and $p^{X^*}(t, z,w)$, respectively, which are continuous in both $z$ and $w$ for all $t>0$.  
The killed semigroup has a heat kernel  $p_{D,V}^{X}(t,z, w)$ satisfying 
\begin{equation}T_{t}^{D,V}f(z) = \int_{D} \,p_{D,V}^{X}(t,z, w) f(w) \, dw.\end{equation}
Inequality (\ref{exit1}) is equivalent to 
\begin{equation} \label{integral}\int_D \,f(w)\, p_{D,V}^{ X}(t,z, w)\,dw \leq
\int_{D^*}\,f^*(w)\,p_{D^*,V^*}^{X^*}(t,0, w)\,dw, \end{equation} 
for all $z\in D$ and all $t>0$, and this in fact holds for all nonnegative Borel functions $f$ by Theorem \ref{fdd1}.   Since $f$ is arbitrary, the continuity assumption of the kernels together with (\ref{Vdomination}) gives that for all $z, w\in D$, 
\begin{equation}
p^{X}_{D,V}(t,z,w)\leq  p^{X^*}_{D^*,V^*}(t,0,0)\leq p^{X^*}_{D^*}(t,0,0)<\infty \label{kernelsinequality}. 
\end{equation}
 If in addition $X_t$ is transient, we can integrate  (\ref{integral}) in time to  obtain the following 
 isoperimetric inequality for the potentials associated to $X_t$ and $X^*_t$.
\begin{corollary} Suppose both $X_t$ and $X_t^*$ are transient and have  continuous densities for all $t>0$. Then for all $z\in D$, 
\begin{equation}\label{green}
\int_D \,f(w) G_{D,V}^{ X}(z, w)\,dw \leq \int_{D^*}\,f^{*}(w) G_{D^*,V^*}^{ X^*}(0,w)\,dw, 
\end{equation}
where $G_{D,V}^{ X}(z, w)$ and $G_{D^*,V^*}^{ X^*}(0,w)$ are the Green's functions corresponding to $X_t$ and $X_t^*$, respectively.  
\end{corollary}

By inequalities (\ref{integral}), (\ref{green}), and Proposition 2.1 in \cite{alvino} (see also page 671 of \cite{BS}), we have 
\begin{corollary} Suppose both $X_t$ and $X_t^*$ are symmetric, transient and have continuous densities for all $t>0$. Then for all increasing convex functions $\Phi:
\bR^+ \to \bR^+$,   
\begin{equation} 
\int_D \Phi \big(\,p^{X}_{D,V}(t,z,w) \,\big)dw \leq
\int_{D^*} \Phi \big(\,p^{X^*}_{D^*,V^*}(t,w,0)\,\big)dw,
\end{equation}
and
\begin{equation}
\int_D \Phi \big(\,G^{X}_{D,V}(z,w) \,\big)dw \leq
\int_{D^*} \Phi \big(\,G^{X^*}_{D^*,V^*}(w,0)\,\big)dw, 
\end{equation}
for all  $z \in D$, $t>0$.  
\end{corollary}
These Corollaries extend several results in C. Bandle \cite{bandle}, see for example  page 214.

The heat kernel $p^{X}_{D,V}(t,z,w)$ can also be represented in terms of the multidimensional distributions.
One easily proves, see \cite{mendez}, that
\begin{eqnarray}\label{bridge} && p^{X}_{D,V}(t,z,w) \\  &=& p^{X}(t,z,w) \,
E^{z}\left\{\,\exp\left[-\int_0^t V(X_s)ds\,\right];\tau_{D}^X>t\, \bigg{|} \,X_t=w\,\right\}.\nonumber
\end{eqnarray}
If $0=t_0< t_1< \ldots < t_m < t$, the conditional finite dimensional distribution
\[P^{z_0}\left\{ X_{t_{1}} \in dz_1,\ldots, X_{t_{m}} \in dz_m \, \bigg{|} \,X_t=w\,\right\},\] 
is given by
\[ \frac{p^X(t-t_{m},z_{m},w)}{ p^X(t,z_{0},w)}\prod_{i=1}^{m}p^X(t_i-t_{i-1},z_{i},z_{i-1}) \,dz_1\ldots dz_m.\]
Combining (\ref{bridge})  with the arguments used in (\ref{0fdd}) we have that
\begin{eqnarray}\label{heat}
&& p^{X}_{D,V}(t,z,w) \\\nonumber
&=& \lim_{m \to \infty}\lim_{k \to \infty} \int_{D_k}\cdots \int_{D_k}\,e^{-\frac{t}{m}\sum_{i=1}^{m}V(X_{\frac{it}{m}})}\, \prod_{i=1}^{m+1}\,p^X\left(\frac{t}{m},z_{i},z_{i-1}
\right)\, \,dz_{1}\cdots dz_{m},
\end{eqnarray}
where $z_0=z$ and $z_{m+1}=w$. 

The proof of  Theorem \ref{fdd} can be adapted to obtain
\begin{equation}\label{trace} \int_D p_{D, V}^{ X}(t,w, w)\,dw \leq
\int_{D^*}p_{D^*, V^*}^{X^*}(t,w, w)\,dw < \infty,  \end{equation}
where the last inequality follows from (\ref{kernelsinequality}) and the fact that $|D^*|<\infty$. 
That is, the trace of the Schr\"odinger semigroup for $H^X_{D,V}$ is maximized 
by the trace of the Schr\"odinger semigroup $H^{X^*}_{D^*,V^*}$.

As explain in \cite{vandenBerg}, the amount of heat contained in the domain 
$D$ at time $t$, when $D$ has temperature 1 at $t=0$ and the
boundary of $D$ is kept at temperature $0$ at all times, is given
by
\[ Q_t(D)= \int_D \int_D p^{B}_{D}(t,z,w)\,dz\,dw,\]
where $B$ is a Brownian motion.   Also  the torsional
rigidity of $D$ is given by
\[\int_{0}^{\infty} Q_t(D)dt=\int_D \int_D G^B_{D}(z,w)\,dz\,dw.\]
Using the representation (\ref{heat}), we obtain the following results for
the heat content and torsional rigidity of L\'evy processes.
\begin{corollary}
Suppose both $X_t$ and $X_t^*$ are transient and have continuous densities for all $t>0$.  Then for all $z\in D$ and $t>0$,
\begin{equation}
\int_D \int_D p^{X}_{D,V}(t,z,w)\,dz\,dw  \leq
\int_{D^*}\int_{D^*} p^{X^*}_{D^*,V^*}(t,z,w)\,dz\,dw,
\label{heatcontent}
\end{equation}
 and
\begin{equation}
\int_D \int_D G^{X}_{D,V}(z,w)\,dz\,dw  \leq
\int_{D^*}\int_{D^*} G^{X^*}_{D^*,V^*}(z,w)\,dz\,dw.
\label{torsional}
\end{equation}
\end{corollary}

We recall  that the semigroup of the process $X_t$ is self-adjoint in $L^2$ if and only if 
the process $X_t$ is symmetric. That is, for any Borel set $A\subset \R^d$, 
\[ P^0\{X_t\in A\}=P^0\{X_t\in -A\}.\] 
In terms of the exponent in the  L\'evy-Khintchine formula this leads to the representation (see \cite{apple})

\[\Psi(\xi)={\frac{1}{2}}\langle \A\cdot\xi, \xi\rangle -\int_{\bR^d}
\left[\,\cos(x\cdot \xi)-1\,\right]\,d\nu(x),\]
where 
$\A$ is a symmetric matrix and $\nu$ is  a symmetric L\'evy measure. That is,  $\left[\nu(A)=\nu(-A)\right]$ for all Borel sets $A$.  
In this case the general theory of Dirichlet forms (see \cite{Da}) guarantees that the Markovian semigroup generated by $X_t$ gives rise to the self-adjoint generator $H^X$.  Recall that $H^X_{V,D}$ is the operator obtained by imposing Dirichlet boundary conditions on $D$ to the Schr\"odinger operator $H^X+V$.  That is, the generator of the  killed semigroup $\{T_t^{D,V}\}_{t \ge 0}$.  By   (\ref{kernelsinequality}) we have that
\begin{eqnarray}\label{finitetrace} \int_D p_{D,V}^{ X}(t,w, w)\,dw &\leq&
\int_{D^*}p_{D^*,V^*}^{X^*}(t,0, 0)\,dw \\ \nonumber 
&=& p_{D^*,V^*}^{X^*}(t,0, 0)\,|D^*|\,< \infty. 
\end{eqnarray}
That is, the semigroup of the killed process has finite trace. 

Whenever  $D$ is of finite volume, the operator $T_t^{D,V}$ maps $L^2(D)$
into $L^{\infty}(D)$ for every $t>0$.  This follows from (\ref{kernelsinequality}) and the general theory of heat semigroups as described on page 59 of  \cite{Da}.  In fact, under these assumptions it follows from 
\cite{Da} that there exists an orthonormal basis of eigenfunctions
$\{\varphi_{D,V,X}^n\}_{n =1}^{\infty}$ for $L^2(D)$ and corresponding eigenvalues $\{\lambda_n ({D,V,X})\}_{n = 1}^{\infty}$ for the semigroup $\{T_t^{D,V}\}_{t \ge 0}$ satisfying
$$0<\lambda_1({D,V,X})<\lambda_2({D,V,X}) \leq \lambda_3({D,V,X})\leq \dots$$
with $\lambda_n({D,V,X}) \to\infty$ as
$n\to\infty$. That is,  the pair $$\{\varphi_{D,V,X}^n, \lambda_n({D,V,X})\}$$ satisfies
\begin{equation*}
T_{t}^{D}\varphi_{D,V,X}^n(z) = e^{-\lambda_n({D,V,X}) t} \,\varphi_{D,V,X}^n(z), \quad z \in D, \,\,\, t > 0.
\end{equation*}
Notice that  $\lambda_n ({D,V,X})$ is a Dirichlet eigenvalue of $H^X+V$ on $D$ with eigenfunction $\varphi_{D,V,X}^n(z)$.
Under such assumptions we have 
\begin{equation}\label{eigenexpan}
p_{D,V}^{X}(t,z, w) = \sum_{n = 1}^{\infty} e^{-\lambda_n({D,V,X}) t}\,\varphi_{D,V,X}^n(z) \,\varphi_{D,V,X}^n(w).
\end{equation}
This eigenfunction expansion for $p_{D,V}^{X}(t,z, w)$ implies  that 
\begin{equation} -\lambda_1({D,V, X})= \lim_{t \to
\infty}{\frac{1}{t}} \log E^{z} \left\{\,\exp\left( -\int_0^t V(X_s)ds\,\right)\,;\,\tau_{D}^X > t\,\right\}, 
\end{equation} for all
domains $D$ of finite volume. This gives the following corollary. 
\begin{corollary}[Faber-Krahn inequality for L\'evy Processes] Suppose both $X_t$ and $X_t^*$ are symmetric, transient and have continuous densities for all $t>0$. Then 
\begin{equation} \lambda_1({D^*,V^*,X^*}) \leq \lambda_1({D,V. 
X}).\label{eigenvalue}
\end{equation}
More generally, we also have the trace inequality 
\[ \sum_{n=1}^\infty e^{-t \lambda_n({D,X, V})} \leq \sum_{n=1}^\infty 
e^{-t \lambda_n({D^*,X^*, V^*})}, \]
valid for all $t>0$.  
\end{corollary}

Finally, denote by  $C_X(A)$ the capacity of the set $A$ for the process $X_t$. In \cite{watanabe}, T. Watanabe proved that 

\begin{equation}C_{X}(A) \geq C_{X^*}(A^*).\label{capacities}\end{equation}
(This question, for Riesz capacities of all orders was raised by P. Mattila in \cite{mattila}.) 
As explained in \cite{mendez1}, this inequality can be obtained from  the existing rearrangement inequalities of multiple integrals only in the case that $X_t$ is isotropic unimodal. For general L\'evy   processes we have 
the following representation  of the capacity due to 
Port and Stone \cite{port} 
\begin{equation}\lim_{t \to \infty} \frac{1}{t} 
\int P^{z_0}\,\left(\, \tau_{A^c}^X \leq t\,\right)\, dz_0= 
C_X(A).\label{23}
\end{equation}
Since 
\begin{eqnarray}
&&\int P^{z_0}\,\left(\, \tau_{A^c}^X \leq t\,\right)\, dz_0\\\nonumber
&=&\lim_{k\to\infty}\lim_{m \to \infty}\int \ldots \int  \Big[1-  \prod_{j=1}^m  I_{A_k^c}(z_j)\:\Big] \; \prod_{j=1}^m
p^X\left(\frac{t}{m},z_{j}, z_{j-1}\right)\,dz_{0}\cdots dz_{m}, 
\end{eqnarray}
where $A_k$ is a decreasing sequence of  compact sets such that the
interior of $A_k$ contains $A$ for all $k$ and
$\cap_{k=1}^{\infty} A_k=A$.  We would expect to obtained (\ref{capacities}) using a result similar to Theorem \ref{bll}
for more general L\'evy processes.
However, the corresponding rearrangement inequality for this type of multiple integrals is only known for 
radially symmetric decreasing functions. That is, only  when $X_t$ is an isotropic unimodal L\'evy process. 
\bigskip

\section{Appendix: Some symmetrization facs} 

We recall once again that the symmetric decreasing rearrangement $f^*$  of $f$ is the  function satisfying 
\[f^*(x)=f^*(y), \text{ if } |x|=|y|,\,\]
\[ f^*(x) \leq f^*(y), \text{ if } |x| \geq  |y|,\] \[ \lim_{|x| \to |y|^+}f^*(x)=f^*(y),
\]
and  
\begin{equation} \label{eq:6.1}
m\left\{ f >  t \right \}= m\left\{ f^* >  t \right\},
\end{equation}
for all $ t \geq 0$. 
Define    $r(f,t)$ as 
\begin{equation}\label{eq:6.2}
m\left\{ f > t \right\}=m\left\{ B\left(\,0,r(f,t)\,\right)\,\right\}.
\end{equation}
Whenever $f$ is a radially  symmetric nonincreasing function such that $f$ is right continuous at $|x_0|$, 
we have that 
\begin{equation}\label{eq:6.3}
r(f,f(x_0))= \sup\{ r >0:  f(r) > f(x_0)\}=|x_0|.
\end{equation}
In particular, given that    $f^*$ is right continuous as a function of the radius, we have  
\[r(f^*,t)= \sup\{ r >0:  f^*(r) > t\},\]
for all $t>0$.

Our first result states that symmetrization preserves continuity and order, see page 81 of \cite{lieb-loss}.

\begin{proposition}\label{order}
Let $f$ be a nonnegative function. If $f$ is continuous, then $f^*$ is  continuous.
In addition, if $g$ is a nonnegative function such that $ g(x) \leq f(x)$ almost everywhere with respect to the Lebesgue measure, then 
\begin{equation}\label{eq:6.4}
g^*(x) \leq f^*(x), 
\end{equation}
for all $x \in \bR$.
\end{proposition}
\begin{proof}
Let us assume that  $f^*(x)$ is not continuous at  $x_0$. Given that $f^*$ is radially symmetric decreasing and right continuous as a function of the radius,  $x_0 \not=0$ and there exist  $t_1$ such that

\[ m\left\{ f^* >  s \right \}=m\left\{ f^* >  f^*(x_0) \right \}\not= 0,\]
for all $ s \in [f^*(x_0),t_1)$. However the continuity of $f$ implies that the set 
\[ \left\{x \in \bR^d: f^*(x_0)<  f(x) <  s  \right\}\]
is  nonempty and open. Therefore 
\[ m\left\{ f>  s \right \}< m\left\{ f >  f^*(x_0) \right \},\]
which is a contradiction. 

Now if $ g(x) \leq f(x)$, almost everywhere with respect to the Lebesgue measure, then 
\begin{eqnarray}\label{order1}
m\left\{ B(0,r\left( g^*,t\,\right)\right\}&=& m\left\{ g > t \right \}\\\nonumber
&\leq& m\left\{ f > t \right \}\\\nonumber
&=& m\left\{ B(0,r\left(\,f^* ,t\,\right)\right\}. 
\end{eqnarray}
That is, 
\begin{equation}\label{eq:6.5}
r\left( g^*,t\,\right)\leq r\left(\,f^* ,t\,\right),
\end{equation}  
for all $t>0$. Let us assume that there exists $x \in \bR^d$ such that $f^*(|x|) < g^*(|x|)$. Since 
$g^*$ is  decreasing and right continuous as a function of the radius, we have
\[  r\left(\,g^* ,g^*(|x|)\,\right)< r\left(\,g^* ,f^*(|x|)\,\right). \]
On the other hand, by (\ref{order1}) and   (\ref{eq:6.5}),
\[
|x|=r\left(\,g^* ,g^*(|x|)\,\right) < r\left(\,g^* ,f^*(|x|)\,\right)
\leq r\left(\,f^* ,f^*(|x|)\,\right)=|x|.
\]

\end{proof}

Finally, we prove that symmetrization preserves almost everywhere convergence.

\begin{proposition}\label{limit-th}
Let $ \{\phi_n\}_{n=1}^\infty$ be a  sequence of bounded  functions such that
\[ \lim_{n \to \infty} \phi_n= \phi,\]
almost everywhere with respect to the Lebesgue measure.
If $x_0$ is a point of continuity of $\phi^*$ then 
\[ \lim_{n \to \infty} \phi_n^*(x_0)= \phi^*(x_0).\] 
In particular
\begin{equation}\label{limit}
 \lim_{n \to \infty} \phi_n^*=\phi^*,
\end{equation}
almost everywhere with respect to the Lebesgue measure.

\end{proposition}
\begin{proof}

Asumme there exists  $x_0$ a continuity point of  $\phi^*$ such that
\[ \lim_{n \to \infty}\phi_n^*(x_0)\not= \phi^*(x_0).\]
Then there exists $\epsilon >0$ and a subsequence $n_k$ such that either 
\begin{equation}\label{ae1}
\phi_{n_k}^*(x_0)>  \phi^*(x_0)+ \epsilon,
\end{equation}
or
\begin{equation}\label{ae2}
\phi_{n_k}^*(x_0)<  \phi^*(x_0) - \epsilon.
\end{equation}
Let us assume that (\ref{ae1}) holds. Since $x_0$ is a continuity point of  $\phi^*$, 
there exists $0< \delta < \epsilon $ and $y_0$ a continuity point of $\phi^*$   such that
\[\phi^{*}(x_0)+\delta=\phi^{*}(y_0), \text { and } |y_0| < |x_0|.\]
However, thanks to (\ref{eq:6.1}) and  (\ref{eq:6.3}), 
\begin{eqnarray*}
m\left\{ B\left(\,0,|x_0|\,\right)\,\right\}&=&\limsup_{ n_k \to \infty} m\left\{ \phi_{n_k}^* > \phi^*_{n_k}(x_0) \right\}\\
& \leq & \limsup_{ n_k \to \infty} m\left\{ \phi_{n_k}^* >  \phi^{*}(x_0)+\delta \right\} \\
& = &m\left\{ \phi^* >   \phi^{*}(y_0)\right\}\\
&= &m\left\{ B\left(\,0,|y_0|\,\right)\,\right\},
\end{eqnarray*}
which is a contradiction. On the other hand, if
\[ \phi_{n_k}^*(x_0)<  \phi^*(x_0) - \epsilon .\]
There exists $0< \delta < \epsilon $ and $y_0$  a continuity point of $\phi^*$  such that
\[\phi^{*}(x_0)-\delta=\phi^{*}(y_0), \text{ and } |y_0| > |x_0|.\]
Since
\begin{eqnarray*}
m\left\{ B\left(\,0,|x_0|\,\right)\,\right\}&=& \limsup_{ n_k \to \infty} m\left\{ \phi_{n_k}^* > \phi_{n_k}^*(x_0) \right\}\\
& \geq& \limsup_{ n_k \to \infty} m\left\{ \phi_{n_k}^* >  \phi^{*}(x_0)-\delta \right\} \\
& =&m\left\{ \phi^* >   \phi^{*}(y_0)\right\}\\
&= & m\left\{ B\left(\,0,|y_0|\,\right)\,\right\},
\end{eqnarray*}
we also obtain a contradiction and this proves the Proposition. 
\end{proof}


\end{document}